\documentclass[10pt,reqno]{amsart}

\usepackage{amsthm, mathrsfs, amsmath, amstext, amsxtra, amsfonts, dsfont, amssymb}
\usepackage{lmodern}
\usepackage[dvipsnames]{xcolor}
\usepackage[colorlinks, linkcolor=red, citecolor=blue, urlcolor=OliveGreen, pagebackref, hypertexnames=false]{hyperref}

\usepackage[belowskip=2pt,aboveskip=0pt]{caption}
\usepackage{booktabs}
\newcommand{\R}{\mathbb R}
\newcommand{\C}{\mathbb C}

\newcommand{\eps}{\epsilon}
\newcommand{\sct} {s_{\text{c}}}
\newcommand{\sce} {s_{\emph{c}}}

\newcommand{\scal}[1]{\left\langle #1 \right\rangle} %scalar product or japanese bracket

\newcommand{\defendproof}{\hfill $\Box$} %end proof notation

\newtheorem{theorem}{Theorem}[section]
\newtheorem{lem}[theorem]{Lemma} 
\newtheorem{prop}[theorem]{Proposition}
\newtheorem{coro}[theorem]{Corollary} 
\theoremstyle{definition}

\newtheorem{rem}[theorem]{Remark}

\title[Blow-up solutions mass-critical NLFS]{On blow-up solutions to the focusing mass-critical nonlinear fractional Schr\"odinger equation} 

\author[V. D. Dinh]{Van Duong Dinh}
\address[V. D. Dinh]{Institut de Math\'ematiques de Toulouse UMR5219, Universit\'e Toulouse CNRS, 31062 Toulouse Cedex 9, France}
\email{dinhvan.duong@math.univ-toulouse.fr}

\keywords{Nonlinear fractional Schr\"odinger equation; Blow-up; Concentration; Limiting profile}
\subjclass[2010]{35B44, 35Q55}

\begin{document}

\maketitle
\begin{abstract}
In this paper we study dynamical properties of blow-up solutions to the focusing mass-critical nonlinear fractional Schr\"odinger equation. We establish a profile decomposition and a compactness lemma related to the equation. As a result, we obtain the $L^2$-concentration and the limiting profile with minimal mass of blow-up solutions. 
\end{abstract}

%\tableofcontents
\section{Introduction}
\setcounter{equation}{0}
Consider the Cauchy problem for nonlinear fractional Schr\"odinger equations
\begin{align}
\left\{
\begin{array}{rcl}
i\partial_t u - (-\Delta)^s u &=& \mu |u|^{\alpha} u, \quad \text{on } [0,+\infty) \times \R^d, \\
u(0) &=& u_0, 
\end{array}
\right.
\label{NLFS}
\end{align}
where $u$ is a complex valued function defined on $[0,+\infty) \times \R^d$, $s \in (0,1)$ and $\alpha>0$. The parameter $\mu=1$ (resp. $\mu=-1$) corresponds to the defocusing (resp. focusing) case. The operator $(-\Delta)^s$ is the fractional Laplacian which is the Fourier multiplier by $|\xi|^{2s}$. The fractional Schr\"odinger equation is a fundamental equation of fractional quantum mechanics, which was discovered by Laskin \cite{Laskin} as a result of extending the Feynmann path integral, from the Brownian-like to L\'evy-like quantum mechanical paths. The fractional Schr\"odinger equation also appears in the continuum limit of discrete models with long-range interactions (see e.g. \cite{KirkLenzStaf}) and in the description of Boson stars as well as in water wave dynamics (see e.g. \cite{FrohJonsLenz} or \cite{IonePusa}). In the last decade, the fractional nonlinear Schr\"odinger equation has attracted a lot of interest in mathematics, numerics and physics (see e.g. \cite{BellazziniGeorgievVisciglia, BoulengerHimmelsbachLenzmann, CaiMajdaMcLaughlinTabak, ChoHajaiejHwangOzawa, ChoHajaiejHwangOzawa-stability, ChoLee, ChoOzawaXia,  ChoOzawa, ChoHwangShim, Dinh-fract, FrankLenzmann, FrankLenzmannSilvestre, Feng, GuoZhu, GuoSireWangZhao, HongSire, IonePusa, KleinSparberMarkowich, KriegerLenzmannRaphael, OzawaVisciglia, SunWangYaoZheng, Zhu, PengShi} and references therein).

The equation $(\ref{NLFS})$ enjoys the scaling invariance
\[
u_\lambda(t,x) = \lambda^{\frac{2s}{\alpha}} u(\lambda^{2s} t, \lambda x), \quad \lambda>0. 
\]
A computation shows
\[
\|u_\lambda(0)\|_{\dot{H}^{\gamma}} = \lambda^{\gamma+\frac{2s}{\alpha} - \frac{d}{2}} \|u_0\|_{\dot{H}^\gamma}.
\]
We thus define the critical exponent
\begin{align}
\sct := \frac{d}{2} -\frac{2s}{\alpha}. \label{critical regularity exponent}
\end{align}
The equation $(\ref{NLFS})$ also enjoys the formal conservation laws for the mass and the energy:
\begin{align*}
M(u(t)) &= \int |u(t,x)|^2 dx = M(u_0), \\
E(u(t)) &= \frac{1}{2} \int |(-\Delta)^{s/2} u(t,x)|^2 dx + \frac{\mu}{\alpha+2} \int |u(t,x)|^{\alpha+2}dx=E(u_0).
\end{align*}
The local well-posedness for $(\ref{NLFS})$ in Sobolev spaces was studied in \cite{HongSire} (see also \cite{ChoHajaiejHwangOzawa} for fractional Hartree equations). Note that the unitary group $e^{-it(-\Delta)^s}$ enjoys several types of Strichartz estimates (see e.g. \cite{ChoOzawaXia} or \cite{Dinh-fract} for Strichartz estimates with non-radial data; and \cite{GuoWang}, \cite{Ke} or \cite{ChoLee} for Strichartz estimates with radially symmetric data; and \cite{FangWang} or \cite{ChoOzawa} for weighted Strichartz estimates). For non-radial data, these Strichartz estimates have a loss of derivatives. This makes the study of local well-posedness more difficult and leads to a weak local theory comparing to the standard nonlinear Schr\"odinger equation (see e.g. \cite{HongSire} or \cite{Dinh-fract}). One can remove the loss of derivatives in Strichartz estimates by considering radially symmetric initial data. However, these Strichartz estimates without loss of derivatives require an restriction on the validity of $s$, that is $s \in \left[\frac{d}{2d-1},1\right)$. We refer the reader to Section $\ref{section preliminaries}$ for more details about Strichartz estimates and the local well-posedness in $H^s$ for $(\ref{NLFS})$. 

Recently, Boulenger-Himmelsbach-Lenzmann \cite{BoulengerHimmelsbachLenzmann} proved blow-up criteria for radial $H^s$ solutions to the focusing $(\ref{NLFS})$. More precisely, they proved the following:
\begin{theorem}[Blow-up criteria \cite{BoulengerHimmelsbachLenzmann}] 
	Let $d\geq 2$, $s \in (1/2, 1)$ and $\alpha>0$. Let $u_0 \in H^s$ be radial and assume that the corresponding solution to the focusing $(\ref{NLFS})$ exists on the maximal time interval $[0,T)$. 
	\begin{itemize}
		\item {\bf Mass-critical case}, i.e. $\sce=0$ or $\alpha = \frac{4s}{d}$: If $E(u_0)<0$, then the solution $u$ either blows up in finite time, i.e. $T<+\infty$ or blows up infinite time, i.e. $T=+\infty$ and 
		\[
		\|u(t)\|_{\dot{H}^s} \geq c t^s, \quad \forall t \geq t_*,
		\]
		with some $C>0$ and $t_*>0$ that depend only on $u_0, s$ and $d$. 
		\item {\bf Mass and energy intercritical case}, i.e. $0<\sce<s$ or $\frac{4s}{d}<\alpha<\frac{4s}{d-2s}$: If $\alpha<4s$ and either
		$E(u_0)<0$, or if $E(u_0) \geq 0$, we assume that
		\[
		E^{\sce}(u_0) M^{s-\sce}(u_0) < E^{\sce}(Q) M^{s-\sce}(Q), \quad \|u_0\|_{\dot{H}^s}^{\sce} \|u_0\|_{L^2}^{s-\sce} > \|Q\|^{\sce}_{\dot{H}^s} \|Q\|_{L^2}^{s-\sce},
		\]
		where $Q$ is the unique (modulo symmetries) positive radial solution to the elliptic equation
		\[
		(-\Delta)^s Q + Q - |Q|^\alpha Q=0,
		\]
		then the solution blows up in finite time, i.e. $T<+\infty$. 
		\item {\bf Energy-critical case}, i.e. $\sce=s$ or $\alpha=\frac{4s}{d-2s}$: If $\alpha<4s$ and either
		$E(u_0)<0$, or if $E(u_0) \geq 0$, we assume that
		\[
		E(u_0) <E(W), \quad \|u_0\|_{\dot{H}^s} > \|W\|_{\dot{H}^s},
		\]
		where $W$ is the unique (modulo symmetries) positive radial solution to the elliptic equation
		\[
		(-\Delta)^s W - |W|^{\frac{4s}{d-2s}} W=0, 
		\]
		then the solution blows up in finite time, i.e. $T<+\infty$. 
	\end{itemize}
\end{theorem}

In this paper we are interested in dynamical properties of blow-up solutions in $H^s$ for the focusing mass-critical nonlinear fractional Schr\"odinger equation, i.e. $s \in (0,1) \backslash \{1/2\}$, $\alpha=\frac{4s}{d}$ and $\mu=-1$ in $(\ref{NLFS})$. Before entering some details of our results, let us recall known results about blow-up solutions in $H^1$ for the focusing mass-critical nonlinear Schr\"odinger equation 
\[
i\partial_t v + \Delta v = -|v|^{\frac{4}{d}} v, \quad v(0) = v_0 \in H^1. \tag{mNLS}
\]
The existence of blow-up solutions in $H^1$ for (mNLS) was firstly proved by Glassey \cite{Glassey}, where the author showed that for any negative energy initial data satisfying $|x| v_0 \in L^2$, the corresponding solution blows up in finite time. Ogawa-Tsutsumi \cite{OgawaTsutsumi, OgawaTsutsumi1d} showed the existence of blow-up solutions for negative energy radial data in dimensions $d\geq 2$ and for any negative energy initial data (without radially symmetry) in the one dimensional case. The study of blow-up $H^1$ solution to (mNLS) is connected to the notion of ground state which is the unique (up to symmetries) positive radial solution to the elliptic equation
\[
\Delta R -R + |R|^{\frac{4}{d}} R =0. 
\]
By the variational characteristic of the ground state, Weinstein \cite{Weinstein} showed the structure and formation of singularity of the minimal mass blow-up solution, i.e. $\|v_0\|_{L^2} = \|R\|_{L^2}$. He proved that the blow-up solution remains close to the ground state $R$ up to scaling and phase parameters, and also translation in the non-radial case. Merle-Tsutsumi \cite{MerleTsutsumi}, Tsutsumi \cite{Tsutsumi} and Nava \cite{Nava} proved the $L^2$-concentration of blow-up solutions by using the variational characterization of ground state, that is, there exists $x(t) \in \R^d$ such that for all $r>0$,
\[
\liminf_{t\uparrow T} \int_{|x-x(t)| \leq r} |v(t,x)|^2 dx \geq \int |R(x)|^2 dx, 
\]
where $T$ is the blow-up time. Merle \cite{Merle92, Merle93} used the conformal invariance and compactness argument to characterize the finite time blow-up solutions with minimal mass. More precisely, he proved that up to symmetries of the equation, the only finite time blow-up solution with minimal mass is the pseudo-conformal transformation of the ground state. Hmidi-Keraani \cite{HmidiKeraani} gave a simplified proof of the characterization of blow-up solutions with minimal mass of Merle by means of the profile decomposition and a refined compactness lemma. Merle-Rapha\"el \cite{MerleRaphael04, MerleRaphael05, MerleRaphael06} established sharp blow-up rates, profiles of blow-up solutions by the help of spectral properties.
		
As for (mNLS), the study of blow-up solution to the focusing mass-critical nonlinear fractional Schr\"odinger equation is closely related to the notion of ground state which is the unique (modulo symmetries) positive radial solution of the elliptic equation
\begin{align}
(-\Delta)^s Q+ Q-|Q|^{\frac{4s}{d}} Q=0. \label{elliptic equation}
\end{align}
The existence and uniqueness (up to symmetries) of ground state $Q \in H^s$ for $(\ref{elliptic equation})$ were recently shown in \cite{FrankLenzmann} and \cite{FrankLenzmannSilvestre}. In \cite{BoulengerHimmelsbachLenzmann, FrankLenzmannSilvestre}, the authors showed the sharp Gagliardo-Nirenberg inequality 
\begin{align}
\|f\|^{\frac{4s}{d}+2}_{L^{\frac{4s}{d}+2}} \leq C_{\text{GN}} \|f\|^{2}_{\dot{H}^s} \|f\|^{\frac{4s}{d}}_{L^2}, \label{sharp gagliardo nirenberg inequality masscritical NLFS}
\end{align}
where
\[
C_{\text{GN}}=\frac{2s+d}{d} \|Q\|^{-\frac{4s}{d}}_{L^2}.
\]
Using this sharp Gagliardo-Nirenberg inequality together with the conservation of mass and energy, it is easy to see that if $u_0 \in H^s$ satisfies
\[
\|u_0\|_{L^2} <\|Q\|_{L^2},
\]
then the corresponding solution exists globally in time. This implies that $\|Q\|_{L^2}$ is the critical mass for the formation of singularities. 

To study blow-up dynamics for data in $H^s$, we establish the profile decomposition for bounded sequences in $H^s$ in the same spirit of \cite{HmidiKeraani}. With the help of this profile decomposition, we prove a compactness lemma related to the focusing mass-critical (NLFS).
\begin{theorem}[Compactness lemma] \label{theorem compactness lemma masscritical NLFS} Let $d\geq 1$ and $0<s<1$. Let $(v_n)_{n\geq 1}$ be a bounded sequence in $H^s$ such that
	\[
	\limsup_{n\rightarrow \infty} \|v_n\|_{\dot{H}^s} \leq M, \quad \limsup_{n\rightarrow \infty} \|v_n\|_{L^{\frac{4s}{d}+2}} \geq m.
	\]
	Then there exists a sequence $(x_n)_{n\geq 1}$ in $\R^d$ such that up to a subsequence,
	\[
	v_n(\cdot + x_n) \rightharpoonup V \text{ weakly in } H^s,
	\]
	for some $V \in H^s$ satisfying
	\begin{align}
	\|V\|^{\frac{4s}{d}}_{L^2} \geq \frac{d}{d+2s} \frac{ m^{\frac{4s}{d}+2}}{M^2} \|Q\|_{L^2}^{\frac{4s}{d}}, \label{lower bound critical sobolev}
	\end{align}
	where $Q$ is the unique solution to the elliptic equation $(\ref{elliptic equation})$.
\end{theorem}
Note that the lower bound on the $L^2$-norm of $V$ is optimal. Indeed, if we take $v_n = Q$, then we get the identity. 

As a consequence of this compactness lemma, we show that the $L^2$-norm of blow-up solutions must concentrate by an amount which is bounded from below by $\|Q\|_{L^2}$ at the blow-up time. More precisely, we prove the following result.
\begin{theorem} [Blow-up concentration] \label{theorem mass concentration NLFS}
	Let 
	\begin{align}
	\renewcommand{\arraystretch}{1.3}
	\left\{
	\begin{array}{l c l l}
	d=1, & s \in \left(\frac{1}{3}, 1 \right) \backslash \left\{\frac{1}{2}\right\}, & \alpha= 4s, & u_0 \in H^s \text{ non-radial}, \\
	d=2, & s\in \left(\frac{1}{2},1\right), & \alpha=2s, & u_0 \in H^s \text{ non-radial}, \\
	d=3, & s\in \left[\frac{3}{5}, \frac{3}{4}\right], & \alpha = \frac{4s}{3} & u_0 \in H^s \text{ radial}, \\
	d=3, & s \in \left(\frac{3}{4},1\right), & \alpha=\frac{4s}{3}, & u_0 \in H^s \text{ non-radial}, \\
	d\geq 4, & s\in \left[\frac{d}{2d-1}, 1\right), & \alpha = \frac{4s}{d}, & u_0 \in H^s \text{ radial}. 
	\end{array}
	\right.
	\label{assumption d s alpha mass concentration NLFS}
	\end{align}
	Assume that the corresponding solution $u$ to $(\ref{NLFS})$ blows up at finite time $0<T<+\infty$. Let $a(t)>0$ be such that 
	\begin{align}
	a(t) \|u(t)\|_{\dot{H}^s}^{\frac{1}{s}} \rightarrow \infty, \label{concentration assumption masscritical NLFS}
	\end{align}
	as $t\uparrow T$. Then there exists $x(t) \in \R^d$ such that
	\begin{align}
	\liminf_{t\uparrow T} \int_{|x-x(t)| \leq a(t)} |u(t,x)|^2 dx \geq \int |Q(x)|^2dx, \label{mass concentration NLFS}
	\end{align}
	where $Q$ is the unique solution to $(\ref{elliptic equation})$.
\end{theorem}
\begin{rem} \label{rem mass concentration NLFS}
	\begin{itemize}
		\item The condition $(\ref{assumption d s alpha mass concentration NLFS})$ comes from the local theory (see Table \ref{table}).
		\item By the blow-up rate given in Corollary $\ref{coro blowup rate NLFS}$, we have 
		\[
		\|u(t)\|_{\dot{H}^s} > \frac{C}{\sqrt{T-t}},
		\]
		for $t\uparrow T$. Rewriting
		\begin{align*}
		\frac{1}{a(t) \|u(t)\|_{\dot{H}^s}^{\frac{1}{s}}} = \frac{\sqrt[2s]{T-t}}{a(t)} \frac{1}{\left(\sqrt{T-t}\|u(t)\|_{\dot{H}^s}\right)^{\frac{1}{s}}} <C\frac{\sqrt[2s]{T-t}}{a(t)},
		\end{align*}
		we see that any function $a(t)>0$ satisfying $\frac{\sqrt[2s]{T-t}}{a(t)} \rightarrow 0$ as $t\uparrow T$ fulfills the conditions of Theorem $\ref{theorem mass concentration NLFS}$. 
	\end{itemize}
\end{rem}
Finally, we show the limiting profile of blow-up solutions with minimal mass $\|Q\|_{L^2}$. More precisely, we show that up to symmetries of the equation, the ground state $Q$ is the profile for blow-up solutions with minimal mass.
\begin{theorem}[Limiting profile with minimal mass] \label{theorem limiting profile minimal mass NLFS}
	Let $d, s, \alpha$ and $u_0$ be as in $(\ref{assumption d s alpha mass concentration NLFS})$. Assume that the corresponding solution $u$ to $(\ref{NLFS})$ blows up at finite time $0<T<+\infty$. If $\|u_0\|_{L^2}=\|Q\|_{L^2}$, then there exist $\theta(t)\in \R$, $\lambda(t)>0$ and $x(t) \in \R^d$ such that 
	\[
	e^{i\theta(t)}  \lambda^{\frac{d}{2}}(t) u(t, \lambda(t) \cdot + x(t)) \rightarrow Q \text{ strongly in } H^s,
	\]
	as $t\uparrow T$.
\end{theorem}

After submitting this manuscript, we are informed that a recent work of Feng \cite{Feng} has considered the fractional nonlinear Schr\"odinger equation with combined power-types of nonlinearities. He studied blow-up dynamics in the case of a $L^2$-critical nonlinear term perturbed by a $L^2$-subcritical term. 

The paper is oganized as follows. In Section $\ref{section preliminaries}$, we recall Strichartz estimates for the fractional Schr\"odinger equation and the local well-posedness for $(\ref{NLFS})$ in non-radial and radial $H^s$ initial data. In Section $\ref{section profile decomposition}$, we show the profile decomposition for bounded sequences in $H^s$ and prove a compactness lemma related to the focusing mass-critical $(\ref{NLFS})$. The $L^2$-concentration of blow-up solutions is proved in Section $\ref{section blowup concentration}$. Finally, we show the limiting profile of blow-up solutions with minimal mass in Section $\ref{section limiting profile}$.
\section{Preliminaries} \label{section preliminaries}
\setcounter{equation}{0}
\subsection{Strichartz estimates}
In this subsection, we recall Strichartz estimates for the fractional Schr\"odinger equation. Let $I \subset \R$ and $p,q \in [1, \infty]$. We define the Strichartz norm
\[
\|f\|_{L^p(I, L^q)} := \Big( \int_I \Big( \int_{\R^d} |f(t,x)|^q dx \Big)^{\frac{p}{q}}\Big)^{\frac{1}{p}},
\]
with a usual modification when either $p$ or $q$ are infinity. We have three-types of Strichartz estimates for the fractional Schr\"odinger equation:
\begin{itemize}
	\item For general data (see e.g. \cite{ChoOzawaXia} or \cite{Dinh-fract}): the following estimates hold for $d\geq 1$ and $s \in (0,1) \backslash \{1/2\}$,
	\begin{align}
	\|e^{-it(-\Delta)^s} \psi\|_{L^p(\R, L^q)} &\lesssim \||\nabla|^{\gamma_{p,q}} \psi\|_{L^2}, \label{strichartz estimate homogeneous}\\
	\Big| \int_0^t e^{-i(t-\tau) (-\Delta)^s} f(\tau) d\tau\Big|_{L^p(\R, L^q)} &\lesssim \||\nabla|^{\gamma_{p,q} -\gamma_{a', b'} - 2s} f\|_{L^{a'}(\R, L^{b'})}, \label{strichartz estimate inhomogeneous}
	\end{align}
	where $(p,q)$ and $(a,b)$ are Schr\"odinger admissible, i.e.
	\[
	p \in [2,\infty], \quad q \in [2, \infty), \quad (p,q, d) \ne (2,\infty, 2), \quad \frac{2}{p} + \frac{d}{q} \leq \frac{d}{2},
	\]
	and 
	\[
	\gamma_{p,q} = \frac{d}{2} -\frac{d}{q} -\frac{2s}{p},
	\]
	similarly for $\gamma_{a', b'}$. Here $(a, a')$ and $(b,b')$ are conjugate pairs. It is worth noticing that for $s \in (0,1) \backslash \{1/2\}$ the admissible condition $\frac{2}{p} +\frac{d}{q} \leq \frac{d}{2}$ implies $\gamma_{p,q}>0$ for all admissible pairs $(p,q)$ except $(p,q)=(\infty, 2)$. This means that the above Strichartz estimates have a loss of derivatives. In the local theory of the nonlinear fractional Schr\"odinger equation, this loss of derivatives makes the problem more difficult, and leads to a weak local well-posedness result comparing to the nonlinear Schr\"odinger equation (see Subsection $\ref{subsection local well posedness}$). 
	\item For radially symmetric data (see e.g. \cite{Ke}, \cite{GuoWang} or \cite{ChoLee}): the estimates $(\ref{strichartz estimate homogeneous})$ and $(\ref{strichartz estimate inhomogeneous})$ hold true for $d\geq 2, s \in (0,1) \backslash \{1/2\}$ and $(p,q), (a,b)$ satisfy the radial Sch\"odinger admissible condition:
	\begin{align*}
	p \in [2, \infty], \quad q \in [2, \infty), \quad (p,q) \ne \left(2, \frac{4d-2}{2d-3}\right), \quad \frac{2}{p} + \frac{2d-1}{q} \leq \frac{2d-1}{2}.
	\end{align*}
	Note that the admissible condition $\frac{2}{p} + \frac{2d-1}{q} \leq \frac{2d-1}{2}$ allows us to choose $(p,q)$ so that $\gamma_{p,q} =0$. More precisely, we have for $d\geq 2$ and $\frac{d}{2d-1}\leq s<1$ \footnote{This condition follows by pluging $\gamma_{p,q}=0$ to $\frac{2}{p}+\frac{2d-1}{q} \leq \frac{2d-1}{2}$.},
	\begin{align}
	\|e^{-it(-\Delta)^s} \psi\|_{L^p(\R, L^q)} &\lesssim \|\psi\|_{L^2}, \label{radial strichartz estimate homogeneous}\\
	\Big| \int_0^t e^{-i(t-\tau) (-\Delta)^s} f(\tau) d\tau\Big|_{L^p(\R, L^q)} &\lesssim \| f\|_{L^{a'}(\R, L^{b'})}, \label{radial strichartz estimate inhomogeneous}
	\end{align}
	where $\psi$ and $f$ are radially symmetric and $(p,q), (a,b)$ satisfy the fractional admissible condition, 
	\begin{align}
	p \in [2,\infty], \quad q \in [2, \infty), \quad (p,q) \ne \left(2, \frac{4d-2}{2d-3}\right), \quad \frac{2s}{p} + \frac{d}{q} = \frac{d}{2}. \label{fractional admissible}
	\end{align}
	These Strichartz estimates with no loss of derivatives allow us to give a similar local well-posedness result as for the nonlinear Schr\"odinger equation (see again Subsection $\ref{subsection local well posedness}$). 
	\item Weighted Strichartz estimates (see e.g. \cite{FangWang} or \cite{ChoOzawa}): for $0<\nu_1<\frac{d-1}{2}$ and $\rho_1 \leq \frac{d-1}{2}-\nu_1$,
	\begin{align}
	\||x|^{\nu_1} |\nabla|^{\nu_1-\frac{d}{2}} \scal{\nabla_\omega}^{\rho_1} e^{-it(-\Delta)^s} \psi\|_{L^\infty_t L^\infty_r L^2_\omega} \lesssim \|\psi\|_{L^2}, \label{weighted strichartz estimate 1}
	\end{align}
	and for $-\frac{d}{2} <\nu_2 <-\frac{1}{2}$ and $\rho_2 \leq -\frac{1}{2}-\nu_2$, 
	\begin{align}
	\||x|^{\nu_2} |\nabla|^{s+\nu_2} \scal{\nabla_\omega}^{\rho_2} e^{-it(-\Delta)^s} \psi\|_{L^2(\R, L^2)} \lesssim \|\psi\|_{L^2}. \label{weighted strichartz estimate 2}
	\end{align}
	Here $\scal{\nabla_\omega} = \sqrt{1-\Delta_\omega}$ with $\Delta_\omega$ is the Laplace-Beltrami operator on the unit sphere $\mathbb{S}^{d-1}$. Here we use the notation 
	\[
	\|f\|_{L^p_r L^q_\omega} = \Big( \int_0^\infty \Big( \int_{\mathbb{S}^{d-1}} |f(r\omega)|^q d\omega\Big)^{p/q} r^{d-1} dr \Big)^{1/p}.
	\]
	These weighted estimates are important to show the well-posedness below $L^2$ at least for the fractional Hartree equation (see \cite{ChoHajaiejHwangOzawa}).
\end{itemize}
\subsection{Nonlinear estimates} 
We recall the following fractional chain rule which is needed in the local well-posedness for $(\ref{NLFS})$. 
\begin{lem}[Fractional chain rule \cite{ChristWeinstein, KenigPonceVega}] \label{lem nonlinear estimate}
	Let $F \in C^1(\C, \C)$ and $\gamma \in (0,1)$. Then for $1<q \leq q_2 <\infty$ and $1<q_1 \leq \infty$ satisfying $\frac{1}{q}=\frac{1}{q_1}+\frac{1}{q_2}$,
	\[
	\||\nabla|^\gamma F(u)\|_{L^q} \lesssim \|F'(u)\|_{L^{q_1}} \||\nabla|^\gamma u\|_{L^{q_2}}.
	\]
\end{lem}
We refer the reader to \cite[Proposition 3.1]{ChristWeinstein} for the proof of the above estimate when $1<q_1<\infty$ and to \cite{KenigPonceVega} for the proof when $q_1=\infty$. 
\subsection{Local well-posedness in $H^s$} \label{subsection local well posedness}
In this section, we recall the local well-posedness in the energy space $H^s$ for $(\ref{NLFS})$. As mentioned in the introduction, we will separate two cases: non-radial initial data and radially symmetric initial data. 
\paragraph{\bf Non-radial $H^s$ initial data.}
We have the following result due to \cite{HongSire} (see also \cite{Dinh-fract}).
\begin{prop}[Non-radial local theory \cite{HongSire, Dinh-fract}] \label{prop local well posedness non radial}
	Let $s \in(0,1) \backslash \{1/2\}$ and $\alpha>0$ be such that 
	\begin{align}
	s>\left\{
	\begin{array}{cl}
	\frac{1}{2} - \frac{2s}{\max (\alpha, 4)} &\text{if } d=1, \\
	\frac{d}{2} - \frac{2s} {\max (\alpha,2)} &\text{if } d\geq 2.
	\end{array}
	\right. \label{subcritical non radial}
	\end{align}
	Then for all $u_0 \in H^s$, there exist $T \in (0,+\infty]$ and a unique solution to $(\ref{NLFS})$ satisfying
	\[
	u \in C([0,T), H^s) \cap L^p_{\emph{loc}}([0,T), L^\infty),
	\]
	for some $p> \max(\alpha, 4)$ when $d=1$ and some $p> \max(\alpha, 2)$ when $d\geq 2$. Moreover, the following properties hold:
	\begin{itemize}
		\item If $T<+\infty$, then $\|u(t)\|_{H^s} \rightarrow \infty$ as $t\uparrow T$.
		\item There is conservation of mass, i.e. $M(u(t)) = M(u_0)$ for all $t \in [0,T)$.
		\item There is conservation of energy, i.e. $E(u(t)) =E(u_0)$ for all $t \in [0,T)$. 
	\end{itemize}	
\end{prop}
The proof of this result is based on Strichartz estimates and the contraction mapping argument. The loss of derivatives in Strichartz estimates can be compensated for by using the Sobolev embedding. We refer the reader to \cite{HongSire} or \cite{Dinh-fract} for more details. 
\begin{rem}
	It follows from $(\ref{subcritical non radial})$ and $s \in (0,1) \backslash \{1/2\}$ that the local well-posedness for non radial data in $H^s$ is available only for 
	\begin{align}
	\left\{
	\begin{array}{lll}
	s \in (1/3, 1/2),& 0<\alpha<\frac{4s}{1-2s} &\text{if } d=1, \\
	s \in (1/2, 1), & 0<\alpha<\infty &\text{if } d=1, \\
	s \in (d/4, 1), & 0<\alpha<\frac{4s}{d-2s} &\text{if } d=2, 3. 
	\end{array}
	\right.
	\label{local well posedness condition}
	\end{align}
	In particular, in the mass-critical case $\alpha=\frac{4s}{d}$, the $(\ref{NLFS})$ is locally well-posed in $H^s$ with
	\[
	\left\{
	\begin{array}{cl}
	s \in (1/3, 1/2) \cap (1/2, 1) &\text{if } d=1, \\
	s \in (d/4, 1) &\text{if } d=2, 3. \\
	\end{array}
	\right.
	\]
\end{rem}
\begin{prop}[Non-radial global existence \cite{Dinh-fract}] \label{prop global existence non radial}
	Let $s, \alpha$ and $d$ be as in $(\ref{local well posedness condition})$. Then for any $u_0 \in H^s$, the solution to $(\ref{NLFS})$ given in Proposition $\ref{prop local well posedness non radial}$ can be extended to the whole $\R$ if one of the following conditions is satisfied:
	\begin{itemize}
		\item $\mu=1$,
		\item $\mu=-1$ and $0<\alpha<\frac{4s}{d}$,
		\item $\mu=-1$, $\alpha=\frac{4s}{d}$ and $\|u_0\|_{L^2}$ is small,
		\item $\mu=-1$ and $\|u_0\|_{H^s}$ is small. 
	\end{itemize}
\end{prop}
\begin{proof}
	The case $\mu=1$ follows easily from the blow-up alternative together with the conservation of mass and energy. The case $\mu=-1$ and $0<\alpha<\frac{4s}{d}$ follows from the Gagliardo-Nirenberg inequality (see e.g. \cite[Appendix]{Tao}). Indeed, by Gagliardo-Nirenberg inequality and the mass conservation, 
	\[
	\|u(t)\|^{\alpha+2}_{L^{\alpha+2}} \lesssim \|u(t)\|_{\dot{H}^s}^{\frac{d\alpha}{2s}} \|u(t)\|_{L^2}^{\alpha+2 - \frac{d\alpha}{2s}} = \|u(t)\|_{\dot{H}^s}^{\frac{d\alpha}{2s}} \|u_0\|_{L^2}^{\alpha+2 - \frac{d\alpha}{2s}}.
	\]
	The conservation of energy then implies
	\[
	\frac{1}{2} \|u(t)\|^2_{\dot{H}^s} = E(u(t)) +\frac{1}{\alpha+2} \|u(t)\|^{\alpha+2}_{L^{\alpha+2}}  \lesssim E(u_0) + \frac{1}{\alpha+2} \|u(t)\|_{\dot{H}^s}^{\frac{d\alpha}{2s}} \|u_0\|_{L^2}^{\alpha+2-\frac{d\alpha}{2s}}. 
	\]
	If $0<\alpha<\frac{4s}{d}$, then $\frac{d\alpha}{2s} \in (0,2)$ and hence $\|u(t)\|_{\dot{H}^s} \lesssim 1$. This combined with the conservation of mass yield the boundedness of $\|u(t)\|_{H^s}$ for any $t$ belongs to the existence time. The blow-up alternative gives the global existence. The case $\mu=-1, \alpha= \frac{4s}{d}$ and $\|u_0\|_{L^2}$ small is treated similarly. It remains to treat the case $\mu=-1$ and $\|u_0\|_{\dot{H}^s}$ is small. Thanks to the Sobolev embedding with $\frac{1}{2} \leq \frac{1}{\alpha+2} + \frac{s}{d}$, we bound
	\[
	\|f\|_{L^{\alpha+2}} \lesssim \|f\|_{H^s}. 
	\]
	This shows in particular that $E(u_0)$ is small if $\|u_0\|_{H^s}$ is small. Therefore,
	\[
	\frac{1}{2} \|u(t)\|^2_{\dot{H}^s} = E(u(t)) +\frac{1}{\alpha+2} \|u(t)\|_{L^{\alpha+2}}^{\alpha+2} \leq E(u_0) + C \|u(t)\|^{\alpha+2}_{H^s}.
	\]
	Since $\|u_0\|_{H^s}$ is small, the above estimate implies that $\|u(t)\|_{H^s}$ is bounded from above and the proof is complete. 
\end{proof}
\paragraph{\bf Radial $H^s$ initial data.} Thanks to Strichartz estimates without loss of derivatives in the radial case, we have the following result.
\begin{prop}[Radial local theory] \label{prop local well posedness radial}
	Let $d\geq 2$ and $s \in \left[\frac{d}{2d-1}, 1\right)$ and $0<\alpha<\frac{4s}{d-2s}$. Let 
	\begin{align}
	p=\frac{4s(\alpha+2)}{\alpha(d-2s)}, \quad q = \frac{d(\alpha+2)}{d+\alpha s}. \label{choice pq radial}
	\end{align}
	Then for any $u_0 \in H^s$ radial, there exist $T\in (0, +\infty]$ and a unique solution to $(\ref{NLFS})$ satisfying 
	\[
	u \in C([0,T), H^s) \cap L^p_{\emph{loc}}([0,T), W^{s, q}).
	\]
	Moreover, the following properties hold:
	\begin{itemize}
		\item If $T<+\infty$, then $\|u(t)\|_{\dot{H}^s} \rightarrow \infty$ as $t\uparrow T$.
		\item $u \in L^a_{\emph{loc}} ([0,T), W^{s,b})$ for any fractional admissible pair $(a,b)$.
		\item There is conservation of mass, i.e. $M(u(t)) = M(u_0)$ for all $t \in [0,T)$.
		\item There is conservation of energy, i.e. $E(u(t)) = E(u_0)$ for all $t\in [0,T)$. 
	\end{itemize}
\end{prop}
\begin{proof}
	It is easy to check that $(p,q)$ satisfies the fractional admissible condition $(\ref{fractional admissible})$. We choose $(m,n)$ so that
	\begin{align}
	\frac{1}{p'} = \frac{1}{p} + \frac{\alpha}{m}, \quad \frac{1}{q'}= \frac{1}{q} + \frac{\alpha}{n}. \label{choice mn}
	\end{align}
	We see that
	\begin{align}
	\frac{\alpha}{m}-\frac{\alpha}{p} = 1-\frac{\alpha (d-2s)}{4s} =:\theta>0, \quad q \leq n = \frac{dq}{d-sq}. \label{define theta}
	\end{align}
	The later fact gives the Sobolev embedding $\dot{W}^{s,q} \hookrightarrow L^n$. Let us now consider
	\[
	X:= \left\{ C(I, H^s) \cap L^p(I, W^{s,q}) \ : \ \|u\|_{L^\infty(I, \dot{H}^s)} + \|u\|_{L^p(I, \dot{W}^{s,q})} \leq M \right\},
	\]
	equipped with the distance
	\[
	d(u,v):= \|u-v\|_{L^\infty(I, L^2)} +\|u-v\|_{L^p(I, L^q)},
	\]
	where $I=[0,\zeta]$ and $M, \zeta>0$ to be chosen later. By Duhamel's formula, it suffices to prove that the functional 
	\[
	\Phi(u)(t) := e^{-it(-\Delta)^s} u_0 -i \mu \int_0^t e^{-i(t-\tau)(-\Delta)^s} |u(\tau)|^\alpha u(\tau) d\tau
	\]
	is a contraction on $(X,d)$. By radial Strichartz estimates $(\ref{radial strichartz estimate homogeneous})$ and $(\ref{radial strichartz estimate inhomogeneous})$,
	\begin{align*}
	\|\Phi(u)\|_{L^\infty(I, \dot{H}^s)} + \|\Phi(u)\|_{L^p(I, \dot{W}^{s, q})} &\lesssim \|u_0\|_{\dot{H}^s} + \||u|^\alpha u\|_{L^{p'}(I, \dot{W}^{s,q'})}, \\
	\|\Phi(u) - \Phi(v)\|_{L^\infty(I, L^2)} + \|\Phi(u) - \Phi(v)\|_{L^p(I, L^q)} &\lesssim \||u|^\alpha u - |v|^\alpha v\|_{L^{p'} (I, L^{q'})}.
	\end{align*}
	The fractional chain rule given in Lemma $\ref{lem nonlinear estimate}$ and the H\"older inequality give
	\begin{align*}
	\||u|^\alpha u\|_{L^{p'}(I, \dot{W}^{s,q'})} &\lesssim \|u\|^\alpha_{L^m(I, L^n)} \|u\|_{L^p(I, \dot{W}^{s,q})}, \\
	&\lesssim |I|^\theta \|u\|^\alpha_{L^p(I, L^n)} \|u\|_{L^p(I, \dot{W}^{s,q})} \\
	&\lesssim |I|^\theta \|u\|^{\alpha+1}_{L^p(I, \dot{W}^{s,q})}.
	\end{align*}
	Similarly,
	\begin{align*}
	\||u|^\alpha- |v|^\alpha v\|_{L^{p'}(I, L^{q'})} &\lesssim \left(\|u\|^\alpha_{L^m(I,L^n)} + \|v\|^\alpha_{L^m(I,L^n)} \right) \|u-v\|_{L^p(I,L^q)} \\
	&\lesssim |I|^\theta \left(\|u\|^\alpha_{L^p(I,\dot{W}^{s,q})} + \|v\|^\alpha_{L^p(I,\dot{W}^{s,q})} \right) \|u-v\|_{L^p(I,L^q)}.
	\end{align*}
	This shows that for all $u, v\in X$, there exists $C>0$ independent of $T$ and $u_0 \in H^s$ such that
	\begin{align*}
	\|\Phi(u)\|_{L^\infty(I, \dot{H}^s)} + \|\Phi(u)\|_{L^p(I, \dot{W}^{s,q})} &\leq C\|u_0\|_{\dot{H}^s} + C \zeta^\theta M^{\alpha+1}, \\
	d(\Phi(u),\Phi(v)) &\leq C\zeta^\theta M^\alpha d(u,v).
	\end{align*}
	If we set $M=2C\|u_0\|_{\dot{H}^s}$ and choose $\zeta>0$ so that
	\[
	C\zeta^\theta M^\alpha \leq \frac{1}{2},
	\]
	then $\Phi$ is a strict contraction on $(X,d)$. This proves the existence of solution $u \in C(I, H^s) \cap L^p(I, W^{s,q})$. By radial Strichartz estimates, we see that $u \in L^a(I, W^{s,b})$ for any fractional admissible pairs $(a,b)$. The blow-up alternative follows easily since the existence time depends only on the $\dot{H}^s$-norm of initial data. The proof is complete.
\end{proof}
As in Proposition $\ref{prop global existence non radial}$, we have the following criteria for global existence of radial solutions in $H^s$. 
\begin{prop}[Radial global existence] \label{prop global existence radial}
	Let $d\geq 2, s \in \left[\frac{d}{2d-1}, 1\right)$ and $0<\alpha<\frac{4s}{d-2s}$. Then for any $u_0\in H^s$ radial, the solution to $(\ref{NLFS})$ given in Proposition $\ref{prop local well posedness radial}$ can be extended to the whole $\R$ if one of the following conditions is satisfied:
	\begin{itemize}
		\item $\mu=1$,
		\item $\mu=-1$ and $0<\alpha<\frac{4s}{d}$,
		\item $\mu=-1$, $\alpha=\frac{4s}{d}$ and $\|u_0\|_{L^2}$ is small, 
		\item $\mu=-1$ and $\|u_0\|_{H^s}$ is small.
	\end{itemize}
\end{prop}
Combining the local well-posedness for non-radial and radial initial data, we obtain the following summary.
\begin{table}[!ht]
	\begin{center}
		\renewcommand{\arraystretch}{1.3}
		\begin{tabular}{|c|c|c|c|}
			\cline{2-4}
			\multicolumn{1}{c|}{} & $s$ & $\alpha$ & LWP \\ \hline
			$d=1$ & $\frac{1}{3}<s<\frac{1}{2}$  & $0<\alpha<\frac{4s}{1-2s}$   & $u_0$ non-radial \\ \hline
			$d=1$ & $\frac{1}{2}<s<1$   & $0<\alpha<\infty$   & $u_0$  non-radial \\ \hline
			$d=2$ & $\frac{1}{2}<s<1$   & $0<\alpha<\frac{4s}{2-2s}$   & $u_0$ non-radial \\ \hline
			$d=3$ & $\frac{3}{5}\leq s\leq \frac{3}{4}$   & $0<\alpha<\frac{4s}{3-2s}$   & $u_0$ radial \\ \hline
			$d=3$ & $\frac{3}{4}<s< 1$   & $0<\alpha<\frac{4s}{3-2s}$   & $u_0$ non-radial  \\ \hline
			$d\geq 4$ & $\frac{d}{2d-1}\leq s<1$   & $0<\alpha<\frac{4s}{d-2s}$   & $u_0$ radial  \\ \hline
		\end{tabular}
	\end{center}
	\caption{Local well-posedness (LWP) in $H^s$ for NLFS}
	\label{table}
\end{table}
\begin{coro}[Blow-up rate] \label{coro blowup rate NLFS}
	Let $d, \alpha$ and $u_0 \in H^s$ be as in Table $\ref{table}$. Assume that the corresponding solution $u$ to $(\ref{NLFS})$ given in Proposition $\ref{prop local well posedness non radial}$ and Proposition $\ref{prop local well posedness radial}$ blows up at finite time $0<T<+\infty$. Then there exists $C>0$ such that
	\begin{align}
	\|u(t)\|_{\dot{H}^s} > \frac{C}{(T-t)^{\frac{s-\sce}{2s}}}, \label{blowup rate NLFS}
	\end{align}
	for all $0<t<T$.
\end{coro}
\begin{proof}
	We follow the argument of Merle-Raphael \cite{MerleRaphael}. Let $0<t<T$ be fixed. We define
	\[
	v_t(\tau, x):= \lambda^{\frac{2s}{\alpha}}(t) u(t+ \lambda^{2s}(t) \tau, \lambda(t) x),
	\]
	with $\lambda(t)$ to be chosen shortly. We see that $v_t$ is well-defined for 
	\[
	t+\lambda^{2s}(t) \tau <T \quad \text{or} \quad \tau < \lambda^{-2s}(t)(T-t).
	\]
	Moreover, $v_t$ solves 
	\[
	i \partial_\tau v_t -(-\Delta)^s v_t =\mu |v_t|^\alpha v_t, \quad v_t(0) = \lambda^{\frac{2s}{\alpha}}(t) u(t, \lambda(t) x).
	\]
	A direct computation shows
	\[
	\|v_t(0)\|_{\dot{H}^s} = \lambda^{s-\sct}(t) \|u(t)\|_{\dot{H}^s}. 
	\]
	Since $s>\sct$, we choose $\lambda(t)$ so that $\|v_t(0)\|_{\dot{H}^s}=1$. Thanks to the local theory, there exists $\tau_0>0$ such that $v_t$ is defined on $[0, \tau_0]$. This shows that
	\[
	\tau_0 < \lambda^{-2s}(t) (T-t) \quad \text{or} \quad \|u(t)\|_{\dot{H}^s} >\frac{\tau_0}{(T-t)^{\frac{s-\sct}{2s}}}.
	\]
	The proof is complete.
\end{proof}
\section{Profile decomposition} \label{section profile decomposition}
\setcounter{equation}{0}
In this subsection, we use the profile decomposition for bounded consequences in $H^s$ to show a compactness lemma related to the focusing mass-critical $(\ref{NLFS})$. 
\begin{theorem}[Profile decomposition] \label{theorem profile decomposition masscritical NLFS}
	Let $d\geq 1$ and $0<s<1$. Let $(v_n)_{n\geq 1}$ be a bounded sequence in $H^s$. Then there exist a subsequence of $(v_n)_{n\geq 1}$ (still denoted $(v_n)_{n\geq 1}$), a family $(x_n^j)_{j\geq 1}$ of sequences in $\R^d$ and a sequence $(V^j)_{j\geq 1}$ of $H^s$ functions such that
	\begin{itemize}
		\item for every $k\ne j$,
		\begin{align}
		|x_n^k - x_n^j| \rightarrow \infty, \quad \text{as } n \rightarrow \infty, \label{pairwise orthogonality masscritical NLFS}
		\end{align}
		\item for every $l\geq 1$ and every $x \in \R^d$,
		\[
		v_n(x) = \sum_{j=1}^l V^j(x-x_n^j) + v_n^l(x),
		\]
		with
		\begin{align}
		\limsup_{n\rightarrow \infty} \|v^l_n\|_{L^q} \rightarrow 0, \quad \text{as } l \rightarrow \infty, \label{profile error masscritical NLFS}
		\end{align}
		for every $q \in (2, 2^\star)$, where 
		\[
		2^\star:= \left\{
		\begin{array}{cl}
		\frac{2}{1-2s} &\text{if } d=1, s \in \left(0,\frac{1}{2}\right),\\
		\infty &\text{if } d=1, s \in \left[\frac{1}{2},1\right), \\
		\frac{2d}{d-2s} &\text{if } d\geq 2, s \in (0,1).
		\end{array}
		\right.
		\]
		Moreover, 
		\begin{align}
		\|v_n\|^2_{L^2} &= \sum_{j=1}^l \|V^j\|^2_{L^2} + \|v^l_n\|^2_{L^2} + o_n(1), \label{profile identity 1 masscritical NLFS} \\
		\|v_n\|^2_{\dot{H}^s} &= \sum_{j=1}^l \|V^j\|^2_{\dot{H}^s} + \|v^l_n\|^2_{\dot{H}^s} + o_n(1), \label{profile identity 2 masscritical NLFS}
		\end{align}
		as $n\rightarrow \infty$.
	\end{itemize}
\end{theorem}
\begin{proof}
	The proof is similar to the one given by Hmidi-Keraani \cite[Proposition 3.1]{HmidiKeraani}. For reader's convenience, we recall some details. 
	Since $H^s$ is a Hilbert space, we denote $\Omega(v_n)$ the set of functions obtained as weak limits of sequences of the translated $v_n(\cdot + x_n)$ with $(x_n)_{n\geq 1}$ a sequence in $\R^d$. Denote
	\[
	\eta(v_n):= \sup \{ \|v\|_{L^2} + \|v\|_{\dot{H}^s} : v \in \Omega(v_n)\}.
	\]
	Clearly,
	\[
	\eta(v_n) \leq \limsup_{n\rightarrow \infty} \|v_n\|_{L^2} + \|v_n\|_{\dot{H}^s}.
	\]
	We shall prove that there exist a sequence $(V^j)_{j\geq 1}$ of $\Omega(v_n)$ and a family $(x_n^j)_{j\geq 1}$ of sequences in $\R^d$ such that for every $k \ne j$,
	\[
	|x_n^k - x_n^j| \rightarrow \infty, \quad \text{as } n \rightarrow \infty,
	\]
	and up to a subsequence, the sequence $(v_n)_{n\geq 1}$ can be written as for every $l\geq 1$ and every $x \in \R^d$,
	\[
	v_n(x) = \sum_{j=1}^l V^j(x-x_n^j) + v^l_n(x), 
	\]
	with $\eta(v^l_n) \rightarrow 0$ as $l \rightarrow \infty$. Moreover, the identities $(\ref{profile identity 1 masscritical NLFS})$ and $(\ref{profile identity 2 masscritical NLFS})$ hold as $n \rightarrow \infty$. \newline
	\indent Indeed, if $\eta(v_n) =0$, then we can take $V^j=0$ for all $j\geq 1$. Otherwise we choose $V^1 \in \Omega(v_n)$ such that
	\[
	\|V^1\|_{L^2} + \|V^1\|_{\dot{H}^s} \geq \frac{1}{2} \eta(v_n) >0. 
	\]
	By the definition of $\Omega(v_n)$, there exists  a sequence $(x^1_n)_{n\geq 1} \subset \R^d$ such that up to a subsequence,
	\[
	v_n(\cdot + x^1_n) \rightharpoonup V^1 \text{ weakly in } H^s.
	\]
	Set $v_n^1(x):= v_n(x) - V^1(x-x^1_n)$. We see that $v^1_n(\cdot + x^1_n) \rightharpoonup 0$ weakly in $H^s$ and
	thus
	\begin{align*}
	\|v_n\|^2_{L^2} &= \|V^1\|^2_{L^2} + \|v^1_n\|^2_{L^2} + o_n(1),  \\
	\|v_n\|^2_{\dot{H}^s} &= \|V^1\|^2_{\dot{H}^s} + \|v^1_n\|^2_{\dot{H}^s} + o_n(1),
	\end{align*}
	as $n \rightarrow \infty$. We now replace $(v_n)_{n\geq 1}$ by $(v^1_n)_{n\geq 1}$ and repeat the same process. If $\eta(v^1_n) =0$, then we choose $V^j=0$ for all $j \geq 2$. Otherwise there exist $V^2 \in \Omega(v^1_n)$ and a sequence $(x^2_n)_{n\geq 1} \subset \R^d$ such that
	\[
	\|V^2\|_{L^2} + \|V^2\|_{\dot{H}^s} \geq \frac{1}{2} \eta(v^1_n)>0,
	\]
	and
	\[
	v^1_n(\cdot+x^2_n) \rightharpoonup V^2 \text{ weakly in } H^s.
	\]
	Set $v^2_n(x) := v^1_n(x) - V^2(x-x^2_n)$. We thus have
	$v^2_n(\cdot +x^2_n) \rightharpoonup 0$ weakly in $H^s$ and 
	\begin{align*}
	\|v^1_n\|^2_{L^2} & = \|V^2\|^2_{L^2} + \|v^2_n\|^2_{L^2} + o_n(1), \\
	\|v^1_n\|^2_{\dot{H}^s} &= \|V^2\|^2_{\dot{H}^s} + \|v^2_n\|^2_{\dot{H}^s} + o_n(1),
	\end{align*}
	as $n \rightarrow \infty$. We claim that 
	\[
	|x^1_n - x^2_n| \rightarrow \infty, \quad \text{as } n \rightarrow \infty.
	\]
	In fact, if it is not true, then up to a subsequence, $x^1_n - x^2_n \rightarrow x_0$ as $n \rightarrow \infty$ for some $x_0 \in \R^d$. Since 
	\[
	v^1_n(x + x^2_n) = v^1_n(x +(x^2_n -x^1_n) + x^1_n),
	\]
	and $v^1_n (\cdot + x^1_n)$ converges weakly to $0$, we see that $V^2=0$. This implies that $\eta(v^1_n)=0$ and it is a contradiction. An argument of iteration and orthogonal extraction allows us to construct the family $(x^j_n)_{j\geq 1}$ of sequences in $\R^d$ and the sequence $(V^j)_{j\geq 1}$ of $H^s$ functions satisfying the claim above. Furthermore, the convergence of the series $\sum_{j\geq 1}^\infty \|V^j\|^2_{L^2} + \|V^j\|^2_{\dot{H}^s}$ implies that 
	\[
	\|V^j\|^2_{L^2} + \|V^j\|^2_{\dot{H}^s} \rightarrow 0, \quad \text{as } j \rightarrow \infty.
	\]
	By construction, we have
	\[
	\eta(v^j_n) \leq 2 \left(\|V^{j+1}\|_{L^2} + \|V^{j+1}\|_{\dot{H}^s}\right),
	\]
	which proves that $\eta(v^j_n) \rightarrow 0$ as $j \rightarrow \infty$. To complete the proof of Theorem $\ref{theorem profile decomposition masscritical NLFS}$, it remains to show $(\ref{profile error masscritical NLFS})$. To do so, we introduce $\theta: \R^d \rightarrow [0,1]$ satisfying $\theta(\xi) =1$ for $|\xi| \leq 1$ and $\theta(\xi) =0$ for $|\xi|\geq 2$. Given $R>0$, define
	\[
	\hat{\chi}_R(\xi) := \theta(\xi/R),
	\]
	where $\hat{\cdot}$ is the Fourier transform of $\chi$. In particular, we have $\hat{\chi}_R (\xi) =1$ if $|\xi| \leq R$ and $\hat{\chi}_R(\xi) =0$ if $|\xi| \geq 2R$. We write
	\[
	v^l_n = \chi_R * v^l_n + (\delta - \chi_R) * v^l_n,
	\]
	where $*$ is the convolution operator. Let $q \in (2, 2^\star)$ be fixed. By Sobolev embedding and the Plancherel formula, we have
	\begin{align*}
	\|(\delta -\chi_R) * v^l_n\|_{L^q} \lesssim \|(\delta-\chi_R) * v^l_n\|_{\dot{H}^\beta} &\lesssim \Big( \int |\xi|^{2\beta} |(1-\hat{\chi}_R(\xi)) \hat{v}^l_n(\xi)|^2 d\xi\Big)^{1/2} \\
	&\lesssim R^{\beta-s} \|v^l_n\|_{H^s},
	\end{align*}
	where $\beta=\frac{d}{2}-\frac{d}{q} \in (0,s)$. On the other hand, the H\"older interpolation inequality implies
	\begin{align*}
	\|\chi_R * v^l_n\|_{L^q} &\lesssim \|\chi_R * v^l_n\|^{\frac{2}{q}}_{L^2} \|\chi_R * v^l_n\|^{1-\frac{2}{q}}_{L^\infty} \\
	&\lesssim \|v^l_n\|^{\frac{2}{q}}_{L^2} \|\chi_R * v^l_n\|^{1-\frac{2}{q}}_{L^\infty}.
	\end{align*}
	Observe that
	\[
	\limsup_{n\rightarrow \infty} \|\chi_R * v^l_n\|_{L^\infty} = \sup_{x_n} \limsup_{n\rightarrow \infty} |\chi_R * v^l_n(x_n)|.
	\]
	Thus, by the definition of $\Omega(v^l_n)$, we infer that
	\[
	\limsup_{n\rightarrow \infty} \|\chi_R * v^l_n\|_{L^\infty} \leq \sup \Big\{ \Big| \int \chi_R(-x) v(x) dx\Big| : v \in \Omega(v^l_n)\Big\}.
	\]
	By the Plancherel formula, we have
	\begin{align*}
	\Big|\int \chi_R(-x) v(x) dx \Big| &= \Big| \int \hat{\chi}_R(\xi) \hat{v}(\xi) d\xi\Big| \lesssim \|\hat{\chi}_R\|_{L^2} \|v\|_{L^2} \\
	&\lesssim R^{\frac{d}{2}} \|\theta\|_{L^2} \|v\|_{L^2} \lesssim R^{\frac{d}{2}} \eta(v^l_n).
	\end{align*}
	We thus obtain for every $l\geq 1$,
	\begin{align*}
	\limsup_{n\rightarrow \infty} \|v^l_n\|_{L^q} &\lesssim \limsup_{n\rightarrow \infty} \|(\delta-\chi_R)* v^l_n\|_{L^q} + \limsup_{n\rightarrow \infty} \|\chi_R * v^l_n\|_{L^q} \\
	&\lesssim R^{\beta-s} \|v^l_n\|_{H^s} + \|v^l_n\|^{\frac{2}{q}}_{L^2} \left[R^{\frac{d}{2}} \eta(v^l_n)\right]^{\left(1-\frac{2}{q}\right)}.
	\end{align*}
	Choosing $R= \left[\eta(v^l_n)^{-1}\right]^{\frac{2}{d}-\eps}$ for some $\eps>0$ small enough, we see that 
	\[
	\limsup_{n\rightarrow \infty} \|v^l_n\|_{L^q} \lesssim \eta(v^l_n)^{(s-\beta)\left(\frac{2}{d}-\eps\right)} \|v^l_n\|_{H^s} + \eta(v^l_n)^{\eps \frac{d}{2} \left(1-\frac{2}{q}\right)} \| v^l_n\|_{L^2}^{\frac{2}{q}}.
	\]
	Letting $l \rightarrow \infty$ and using the fact that $\eta(v^l_n) \rightarrow 0$ as $l \rightarrow \infty$ and the uniform boundedness in $H^s$ of $(v^l_n)_{l\geq 1}$, we obtain 
	\[
	\limsup_{n \rightarrow \infty} \|v^l_n\|_{L^q} \rightarrow 0, \quad \text{as } l \rightarrow \infty.
	\]
	The proof is complete.
\end{proof} 
We are now able to give the proof of the concentration compactness lemma given in Theorem $\ref{theorem compactness lemma masscritical NLFS}$. 

\noindent \textit{Proof of Theorem $\ref{theorem compactness lemma masscritical NLFS}$.} According to Theorem $\ref{theorem profile decomposition masscritical NLFS}$, there exist a sequence $(V^j)_{j\geq 1}$ of $H^s$ functions and a family $(x^j_n)_{j\geq 1}$ of sequences in $\R^d$ such that up to a subsequence, the sequence $(v_n)_{n\geq 1}$ can be written as
\[
v_n(x) = \sum_{j=1}^l V^j(x-x^j_n) + v^l_n(x),
\]
and $(\ref{profile error masscritical NLFS})$, $(\ref{profile identity 1 masscritical NLFS}), (\ref{profile identity 2 masscritical NLFS})$ hold. This implies that
\begin{align}
m^{\frac{4s}{d}+2} &\leq \limsup_{n\rightarrow \infty} \|v_n\|_{L^{\frac{4s}{d}+2}}^{\frac{4s}{d}+2} = \limsup_{n\rightarrow \infty} \Big\| \sum_{j=1}^l V^j(\cdot -x^j_n) + v^l_n\Big\|^{\frac{4s}{d}+2}_{L^{\frac{4s}{d}+2}} \nonumber \\
&\leq \limsup_{n\rightarrow \infty} \Big( \Big\|\sum_{j=1}^l V^j(\cdot -x^j_n) \Big\|_{L^{\frac{4s}{d}+2}} + \|v^l_n\|_{L^{\frac{4s}{d}+2}}\Big)^{\frac{4s}{d}+2} \nonumber \\
&\leq \limsup_{n\rightarrow \infty} \Big\| \sum_{j=1}^\infty V^j(\cdot -x^j_n)\Big\|_{L^{\frac{4s}{d}+2}}^{\frac{4s}{d}+2}. \label{compactness lemma proof masscritical NLFS}
\end{align}
By the elementary inequality 
\[
\left| \Big| \sum_{j=1}^l a_j\Big|^{\frac{4s}{d} +2} - \sum_{j=1}^l |a_j|^{\frac{4s}{d}+2}\right| \leq C \sum_{j \ne k} |a_j| |a_k|^{\frac{4s}{d}+1},
\]
we have
\begin{align*}
\int \Big| \sum_{j=1}^l V^j(x -x^j_n)\Big|^{\frac{4s}{d}+2} dx &\leq \sum_{j=1}^l \int |V^j(x-x^j_n)|^{\frac{4s}{d}+2} dx + C \sum_{j\ne k} \int |V^j(x-x^j_n)||V^k(x-x^k_n)|^{\frac{4s}{d}+1} dx \\
&\leq \sum_{j=1}^l \int |V^j(x-x^j_n)|^{\frac{4s}{d}+2} dx +  C \sum_{j \ne k} \int |V^j(x+ x^k_n-x^j_n)| |V^k(x)|^{\frac{4s}{d}+1} dx.
\end{align*}
Using the pairwise orthogonality $(\ref{pairwise orthogonality masscritical NLFS})$, the H\"older inequality implies that $V^j(\cdot + x^k_n-x^j_n) \rightharpoonup 0$ in $H^s$ as $n \rightarrow \infty$ for any $j \ne k$. This leads to the mixed terms in the sum $(\ref{compactness lemma proof masscritical NLFS})$ vanish as $n\rightarrow \infty$. We thus get
\[
m^{\frac{4s}{d}+2} \leq \sum_{j=1}^\infty \|V^j\|_{L^{\frac{4s}{d}+2}}^{\frac{4s}{d}+2}.
\]
We next use the sharp Gagliardo-Nirenberg inequality $(\ref{sharp gagliardo nirenberg inequality masscritical NLFS})$ to estimate
\begin{align}
\sum_{j=1}^\infty \|V^j\|^{\frac{4s}{d}+2}_{L^{\frac{4s}{d}+2}} \leq \frac{2s+d}{d} \|Q\|_{L^2}^{-\frac{4s}{d}} \sup_{j\geq 1} \|V^j\|^{\frac{4s}{d}}_{L^2} \sum_{j=1}^\infty \|V^j\|^2_{\dot{H}^s}. \label{compactness lemma proof 1 masscritical NLFS}
\end{align}
By $(\ref{profile identity 2 masscritical NLFS})$, we infer that
\[
\sum_{j=1}^\infty \|V^j\|^2_{\dot{H}^s}  \leq \limsup_{n\rightarrow \infty} \|v_n\|^2_{\dot{H}^s} \leq M^2.
\]
Therefore,
\[
\sup_{j\geq 1} \|V^j\|^{\frac{4s}{d}}_{L^2} \geq \frac{d}{d+2s} \frac{m^{\frac{4s}{d}+2}}{M^2} \|Q\|_{L^2}^{\frac{4s}{d}}.
\]
Since the series $\sum_{j\geq 1} \|V^j\|^2_{L^2}$ is convergent, the supremum above is attained. In particular, there exists $j_0$ such that
\[
\|V^{j_0}\|^{\frac{4s}{d}}_{L^2} \geq \frac{d}{d+2s} \frac{m^{\frac{4s}{d}+2}}{M^2} \|Q\|_{L^2}^{\frac{4s}{d}}.
\]
By a change of variables, we write
\[
v_n(x+ x^{j_0}_n) = V^{j_0} (x) + \sum_{1\leq j \leq l \atop j \ne j_0} V^j(x+ x_n^{j_0} - x^j_n) + \tilde{v}^l_n(x),
\]
where $\tilde{v}^l_n(x):= v^l_n(x+x^{j_0}_n)$. The pairwise orthogonality of the family $(x_n^j)_{j\geq 1}$ implies
\[
V^j( \cdot +x^{j_0}_n -x^j_n) \rightharpoonup 0 \text{ weakly in } H^s,
\]
as $n \rightarrow \infty$ for every $j \ne j_0$. We thus get
\begin{align}
v_n(\cdot + x^{j_0}_n) \rightharpoonup V^{j_0} + \tilde{v}^l, \quad \text{as } n \rightarrow \infty, \label{compactness lemma proof 2 masscritical NLFS}
\end{align}
where $\tilde{v}^l$ is the weak limit of $(\tilde{v}^l_n)_{n\geq 1}$. On the other hand, 
\[
\|\tilde{v}^l\|_{L^{\frac{4s}{d}+2}} \leq \limsup_{n\rightarrow \infty} \|\tilde{v}^l_n\|_{L^{\frac{4s}{d}+2}} = \limsup_{n\rightarrow \infty} \|v^l_n\|_{L^{\frac{4s}{d}+2}} \rightarrow 0, \quad \text{as } l \rightarrow \infty. 
\]
By the uniqueness of the weak limit $(\ref{compactness lemma proof 2 masscritical NLFS})$, we get $\tilde{v}^l=0$ for every $l \geq j_0$. Therefore, we obtain 
\[
v_n(\cdot + x^{j_0}_n) \rightharpoonup V^{j_0}.
\]
The sequence $(x^{j_0}_n)_{n\geq 1}$ and the function $V^{j_0}$ now fulfill the conditions of Theorem $\ref{theorem compactness lemma masscritical NLFS}$. The proof is complete.
\defendproof 

\section{Blow-up concentration} \label{section blowup concentration}
\setcounter{equation}{0}
In this section, we give the proof of the mass-concentration of finite time blow-up solutions given in Theorem $\ref{theorem mass concentration NLFS}$. 

\noindent \textit{Proof of Theorem $\ref{theorem mass concentration NLFS}$.}
Let $(t_n)_{n\geq 1}$ be a sequence such that $t_n \uparrow T$. Set
\[
\lambda_n := \left(\frac{\|Q\|_{\dot{H}^s}}{\|u(t_n)\|_{\dot{H}^s}}\right)^{\frac{1}{s}}, \quad v_n(x):= \lambda_n^{\frac{d}{2}} u(t_n, \lambda_n x). 
\]
By the blow-up alternative, we see that $\lambda_n \rightarrow 0$ as $n \rightarrow \infty$. Moreover, we have
\begin{align*}
\|v_n\|_{L^2} = \|u(t_n)\|_{L^2} =\|u_0\|_{L^2},
\end{align*}
and 
\[
\|v_n\|_{\dot{H}^s} = \lambda_n^s\|u(t_n)\|_{\dot{H}^s}= \|Q\|_{\dot{H}^s},
\] 
and
\[
E(v_n) = \lambda_n^{2s} E(u(t_n)) = \lambda_n^{2s} E(u_0) \rightarrow 0, \quad \text{as } n \rightarrow \infty. 
\]
This implies in particular that
\begin{align*}
\|v_n\|^{\frac{4s}{d}+2}_{L^{\frac{4s}{d}+2}} \rightarrow \frac{d+2s}{d} \|Q\|^2_{\dot{H}^s}, \quad \text{as } n \rightarrow \infty. 
\end{align*}
The sequence $(v_n)_{n\geq 1}$ satisfies the conditions of Theorem $\ref{theorem compactness lemma masscritical NLFS}$ with 
\[
m^{\frac{4s}{d}+2} = \frac{d+2s}{d} \|Q\|^2_{\dot{H}^s}, \quad M^2 = \|Q\|^2_{\dot{H}^s}. 
\]
Therefore, there exists a sequence $(x_n)_{n\geq 1}$ in $\R^d$ such that up to a subsequence,
\[
v_n(\cdot + x_n)  = \lambda_n^{\frac{d}{2}} u(t_n, \lambda_n \cdot + x_n) \rightharpoonup V \text{ weakly in } H^s,
\]
as $n \rightarrow \infty$ with $\|V\|_{L^2} \geq \|Q\|_{L^2}$. In particular, 
\[
v(\cdot + x_n) = \lambda_n^{\frac{d}{2}} u(t_n, \lambda_n \cdot + x_n) \rightharpoonup V \text{ weakly in } L^2. 
\]
This implies for every $R>0$,
\[
\liminf_{n\rightarrow \infty} \int_{|x|\leq R} \lambda_n^{d}| u(t_n, \lambda_n x + x_n)|^2 dx \geq \int_{|x|\leq R} |V(x)|^2 dx,
\]
or 
\[
\liminf_{n\rightarrow \infty} \int_{|x-x_n|\leq R\lambda_n} |u(t_n, x)|^2 dx \geq \int_{|x|\leq R} |V(x)|^2 dx.
\]
Since 
\[
a(t_n) \|u(t_n)\|^{\frac{1}{s}}_{\dot{H}^s} = \frac{a(t_n)}{\lambda_n} \|Q\|^{\frac{1}{s}}_{\dot{H}^s},
\]
the assumption $(\ref{concentration assumption masscritical NLFS})$ implies $\frac{a(t_n)}{\lambda_n} \rightarrow \infty$ as $n\rightarrow \infty$. We thus get
\[
\liminf_{n\rightarrow \infty} \sup_{y\in \R^d} \int_{|x-y|\leq a(t_n)} |u(t_n, x)|^2 dx \geq \int_{|x|\leq R} |V(x)|^2 dx,
\]
for every $R>0$, which means that
\[
\liminf_{n\rightarrow \infty} \sup_{y \in \R^d} \int_{|x-y|\leq a(t_n)} |u(t_n, x)|^2 dx \geq \int |V(x)|^2 dx \geq  \int |Q(x)|^2 dx.
\]
Since the sequence $(t_n)_{n\geq 1}$ is arbitrary, we infer that
\[
\liminf_{t\uparrow T} \sup_{y\in \R^d} \int_{|x-y|\leq a(t)} |u(t,x)|^2 dx \geq \int |Q(x)|^2 dx.
\]
But for every $t \in (0,T)$, the function $y\mapsto \int_{|x-y| \leq a(t)} |u(t,x)|^2 dx$ is continuous and goes to zero at infinity. As a result, we get
\[
\sup_{y\in \R^d} \int_{|x-y|\leq a(t)} |u(t,x)|^2 dx = \int_{|x-x(t)| \leq a(t)} |u(t,x)|^2 dx,
\]
for some $x(t) \in \R^d$. This shows $(\ref{mass concentration NLFS})$. The proof is complete.
\defendproof
\section{Limiting profile with minimal mass} \label{section limiting profile}
\setcounter{equation}{0}
In this section, we give the proof of the limiting profile given in Theorem $\ref{theorem limiting profile minimal mass NLFS}$. Let us start with the following characterization of solution with minimal mass.
\begin{lem} \label{lem characterization minimal mass NLFS}
	Let $d\geq 1$ and $0<s<1$. If $u \in H^s$ is such that $\|u\|_{L^2}=\|Q\|_{L^2}$ and $E(u)=0$, then $u$ is of the form
	\[
	u(x) = e^{i\theta} \lambda^{\frac{d}{2}} Q(\lambda x + x_0),
	\]
	for some $\theta \in \R, \lambda>0$ and $x_0 \in \R^d$. 
\end{lem}
\begin{proof}
	Since $E(u)=0$, we have
	\[
	\|u\|^2_{\dot{H}^s} = \frac{d}{d+2s} \|u\|^{\frac{4s}{d}+2}_{L^{\frac{4s}{d}+2}}.
	\]
	Thus
	\[
	\frac{\|u\|^{\frac{4s}{d}+2}_{L^{\frac{4s}{d}+2}}}{\|u\|^{\frac{4s}{d}}_{L^2} \|u\|^2_{\dot{H}^s}} = \frac{d+2s}{d} \|u\|^{-\frac{4s}{d}}_{L^2} = \frac{d+2s}{d} \|Q\|_{L^2}^{-\frac{4s}{d}} = C_{\text{GN}},
	\]
	where $C_{\text{GN}}$ is the sharp constant in $(\ref{sharp gagliardo nirenberg inequality masscritical NLFS})$. By the characterization of the sharp constant to the Gagliardo-Nirenberg inequality $(\ref{sharp gagliardo nirenberg inequality masscritical NLFS})$ (see e.g. \cite[Section 3]{FrankLenzmannSilvestre}), we learn that $u$ is of the form $u(x) = a Q(\lambda x +x_0)$ for some $a \in \C^\star, \lambda>0$ and $x_0 \in \R^d$. On the other hand, since $\|u\|_{L^2} =\|Q\|_{L^2}$, we have $|a|= \lambda^{\frac{d}{2}}$. This shows the result. 
\end{proof}
We are now able to prove the limiting profile of finite time blow-up solutions with minimal mass given in Theorem $\ref{theorem limiting profile minimal mass NLFS}$. 

\noindent \textit{Proof of Theorem $\ref{theorem limiting profile minimal mass NLFS}$.}
	We will show that for any $(t_n)_{n\geq 1}$ satisfying $t_n \uparrow T$, there exist a subsequence still denoted by $(t_n)_{n\geq 1}$, sequences of $\theta_n \in \R, \lambda_n>0$ and $x_n \in \R^d$ such that
	\begin{align}
	e^{it\theta_n} \lambda^{\frac{d}{2}}_n u(t_n, \lambda_n \cdot + x_n) \rightarrow Q \text{ strongly in } H^s \text{ as } n \rightarrow \infty. \label{limiting profile minimal mass NLFS}
	\end{align}
	Let $(t_n)_{n\geq 1}$ be a sequence such that $t_n \uparrow T$. Set
	\[
	\lambda_n := \left(\frac{\|Q\|_{\dot{H}^s}}{\|u(t_n)\|_{\dot{H}^s}}\right)^{\frac{1}{s}}, \quad v_n(x):= \lambda_n^{\frac{d}{2}} u(t_n, \lambda_n x).
	\]
	By the blow-up alternative, we see that $\lambda_n \rightarrow 0$ as $n \rightarrow \infty$. Moreover, we have
	\begin{align}
	\|v_n\|_{L^2} = \|u(t_n)\|_{L^2} = \|u_0\|_{L^2}=\|Q\|_{L^2},  \label{property v_n masscritical NLFS}
	\end{align}
	and
	\begin{align}
	\|v_n\|_{\dot{H}^s} = \lambda_n^s \|u(t_n)\|_{\dot{H}^s} = \| Q\|_{\dot{H}^s}, \label{property v_n masscritical 1 NLFS}
	\end{align}
	and
	\[
	E(v_n) = \lambda_n^{2s} E(u(t_n)) = \lambda_n^{2s} E(u_0) \rightarrow 0, \quad \text{as } n \rightarrow \infty. 
	\]
	This yields in particular that
	\begin{align}
	\|v_n\|^{\frac{4s}{d}+2}_{L^{\frac{4s}{d}+2}} \rightarrow \frac{d+2s}{d} \|Q\|^2_{\dot{H}^s}, \quad \text{as } n \rightarrow \infty. \label{convergence v_n masscritical NLFS}
	\end{align}
	The sequence $(v_n)_{n\geq 1}$ satisfies the conditions of Theorem $\ref{theorem compactness lemma masscritical NLFS}$ with 
	\[
	m^{\frac{4s}{d}+2} = \frac{d+2s}{d} \|Q\|^2_{\dot{H}^s}, \quad M^2 = \|Q\|^2_{\dot{H}^s}. 
	\]
	Therefore, there exists a sequence $(x_n)_{n\geq 1}$ in $\R^d$ such that up to a subsequence,
	\[
	v_n(\cdot + x_n)  = \lambda_n^{\frac{d}{2}} u(t_n, \lambda_n \cdot + x_n) \rightharpoonup V \text{ weakly in } H^s,
	\]
	as $n \rightarrow \infty$ with $\|V\|_{L^2} \geq \|Q\|_{L^2}$. 
	Since $v_n(\cdot +x_n) \rightharpoonup V$ weakly in $H^s$ as $n\rightarrow \infty$, the semi-continuity of weak convergence and $(\ref{property v_n masscritical NLFS})$ imply
	\[
	\|V\|_{L^2} \leq \liminf_{n\rightarrow \infty} \|v_n\|_{L^2} \leq \|Q\|_{L^2}. 
	\]
	This together with the fact $\|V\|_{L^2} \geq \|Q\|_{L^2}$ show that
	\begin{align}
	\|V\|_{L^2} = \|Q\|_{L^2} = \lim_{n\rightarrow \infty} \|v_n\|_{L^2}. \label{l2 norm v_n masscritical NLFS}
	\end{align}
	Therefore
	\[
	v_n(\cdot +x_n) \rightarrow V \text{ strongly in } L^2 \text{ as } n\rightarrow \infty. 
	\]
	On the other hand, the Gagliardo-Nirenberg inequality $(\ref{sharp gagliardo nirenberg inequality masscritical NLFS})$ shows that $v_n(\cdot +x_n) \rightarrow V$ strongly in $L^{\frac{4s}{d}+2}$ as $n \rightarrow \infty$. Indeed, by $(\ref{property v_n masscritical 1 NLFS})$,
	\begin{align*}
	\|v_n(\cdot +x_n) -V\|^{\frac{4s}{d}+2}_{L^{\frac{4s}{d}+2}} &\lesssim \|v_n(\cdot + x_n) - V\|^{\frac{4s}{d}}_{L^2} \|v_n(\cdot + x_n) - V\|^2_{\dot{H}^s} \\
	&\lesssim (\|Q\|_{\dot{H}^s} + \|V\|_{\dot{H}^s})^2 \|v_n(\cdot +x_n) -V\|^{\frac{4s}{d}}_{L^2} \rightarrow 0,
	\end{align*}
	as $n\rightarrow \infty$. Moreover, using $(\ref{convergence v_n masscritical NLFS})$ and $(\ref{l2 norm v_n masscritical NLFS})$, the sharp Gagliardo-Nirenberg inequality $(\ref{sharp gagliardo nirenberg inequality masscritical NLFS})$ yields
	\begin{align*}
	\|Q\|^2_{\dot{H}^s} = \frac{d}{d+2s} \lim_{n\rightarrow \infty} \|v_n\|^{\frac{4s}{d}+2}_{L^{\frac{4s}{d}+2}} = \frac{d}{d+2s} \|V\|^{\frac{4s}{d}+2}_{L^{\frac{4s}{d}+2}} \leq \Big(\frac{\|V\|_{L^2}}{\|Q\|_{L^2}} \Big)^{\frac{4s}{d}} \|V\|^2_{\dot{H}^s} = \|V\|^2_{\dot{H}^s},
	\end{align*}
	or $\|Q\|_{\dot{H}^s} \leq \|V\|_{\dot{H}^s}$. By the semi-continuity of weak convergence and $(\ref{property v_n masscritical 1 NLFS})$, 
	\[
	\|V\|_{\dot{H}^s} \leq \liminf_{n\rightarrow \infty} \|v_n\|_{\dot{H}^s} = \|Q\|_{\dot{H}^s}. 
	\]
	Therefore,
	\begin{align}
	\|V\|_{\dot{H}^s} = \|Q\|_{\dot{H}^s} = \lim_{n\rightarrow \infty} \|v_n\|_{\dot{H}^s}.\label{H dot s norm v_n masscritical NLFS}
	\end{align}
	Combining $(\ref{l2 norm v_n masscritical NLFS}), (\ref{H dot s norm v_n masscritical NLFS})$ and using the fact $v_n(\cdot + x_n) \rightharpoonup V$ weakly in $H^s$, we conclude that 
	\[
	v_n(\cdot + x_n) \rightarrow V \text{ strongly in } H^s \text{ as } n\rightarrow \infty.
	\]
	In particular, we have
	\[
	E(V) = \lim_{n\rightarrow \infty} E(v_n) =0.
	\]
	This shows that there exists $V \in H^s$ such that
	\[
	\|V\|_{L^2} = \|Q\|_{L^2}, \quad E(V) =0.
	\]
	By Lemma $\ref{lem characterization minimal mass NLFS}$, we have $V(x) = e^{i\theta} \lambda^{\frac{d}{2}} Q(\lambda x +x_0)$ for some $\theta \in \R, \lambda>0$ and $x_0 \in \R^d$. Thus
	\[
	v_n(\cdot + x_n) = \lambda_n^{\frac{d}{2}} u(t_n, \lambda_n \cdot + x_n) \rightarrow V = e^{i\theta} \lambda^{\frac{d}{2}} Q(\lambda \cdot +x_0) \text{ strongly in } H^s \text{ as } n \rightarrow \infty.
	\]
	Redefining variables as
	\[
	\tilde{\lambda}_n:= \lambda_n \lambda^{-1}, \quad \tilde{x}_n:= \lambda_n \lambda^{-1} x_0 +x_n,
	\]
	we get
	\[
	e^{-i\theta} \tilde{\lambda}^{\frac{d}{2}}_n u(t_n, \tilde{\lambda}_n \cdot + \tilde{x}_n) \rightarrow Q \text{ strongly in } H^s \text{ as } n\rightarrow \infty.
	\]
	This proves $(\ref{limiting profile minimal mass NLFS})$ and the proof is complete.
\defendproof

\section*{Acknowledgments}
The author would like to express his deep gratitude to his wife - Uyen Cong for her encouragement and support. He would like to thank his supervisor Prof. Jean-Marc Bouclet for the kind guidance and constant encouragement. He also would like to thank the reviewer for his/her helpful comments and suggestions. 

% Bibliography-----------------------------------------------------------------

\end{document}